\documentclass[10pt]{article}

\usepackage[letterpaper, left=1in, top=1in, right=1in, bottom=1in, verbose, ignoremp]{geometry}

\usepackage{url}\RequirePackage[colorlinks,citecolor=blue, linkcolor=blue,urlcolor = blue]{hyperref}
\usepackage{latexsym,amssymb,amsmath,amsfonts,graphicx,color,fancyvrb,amsthm,enumerate,mathrsfs}
\usepackage[authoryear,round]{natbib}
\usepackage[dvipsnames]{xcolor}
\usepackage{xy}\xyoption{all} \xyoption{poly} \xyoption{knot}
\usepackage{float}
\usepackage{bm}
\usepackage{bbm}
\usepackage{multicol,multirow}
\usepackage{array}
\usepackage{relsize}
\usepackage{chngcntr}
\usepackage{etoolbox}
\usepackage{caption,subcaption}
\usepackage{tikz}
\usetikzlibrary{decorations.pathreplacing,calc}
\usetikzlibrary{shapes,backgrounds}
\usetikzlibrary{patterns}
\usepackage{amsopn}
\usepackage{calc}
\usepackage{algorithm}
\usepackage[noend]{algpseudocode}
\usepackage{hyperref}
\usepackage{mathtools}

\newtheorem{theorem}{Theorem}[]
\newtheorem*{theorem*}{Theorem}
\newtheorem{corollary}[theorem]{Corollary}
\newtheorem{lemma}[theorem]{Lemma}
\newtheorem{proposition}[theorem]{Proposition}

\newtheorem*{claim*}{Claim}
\theoremstyle{definition}
\newtheorem{definition}[theorem]{Definition}
\newtheorem*{definition*}{Definition}
\theoremstyle{AppDefinition}

\theoremstyle{AppClaim}

\theoremstyle{remark}
\newtheorem{remark}[theorem]{Remark}

\newtheorem{example}[theorem]{Example}
\newtheorem*{example*}{Example}

\def\log{{\rm log}}

\newcommand*\diff{\mathop{}\!\mathrm{d}}

\makeatletter
\newcommand*{\op}{%
  \DOTSB
  \mathop{\vphantom{\bigoplus}\mathpalette\matt@op\relax}%
  \slimits@
}
\newcommand\matt@op[2]{%
  \vcenter{\m@th\hbox{\resizebox{\widthof{$#1\bigoplus$}}{!}{$\boxplus$}}}%
}
\makeatother

\def\dsum{\displaystyle\sum}

\makeatletter
\def\@biblabel#1{}
\makeatother

\makeatletter
\patchcmd{\NAT@citex}
  {\@citea\NAT@hyper@{%
     \NAT@nmfmt{\NAT@nm}%
     \hyper@natlinkbreak{\NAT@aysep\NAT@spacechar}{\@citeb\@extra@b@citeb}%
     \NAT@date}}
  {\@citea\NAT@nmfmt{\NAT@nm}%
   \NAT@aysep\NAT@spacechar\NAT@hyper@{\NAT@date}}{}{}

\patchcmd{\NAT@citex}
  {\@citea\NAT@hyper@{%
     \NAT@nmfmt{\NAT@nm}%
     \hyper@natlinkbreak{\NAT@spacechar\NAT@@open\if*#1*\else#1\NAT@spacechar\fi}%
       {\@citeb\@extra@b@citeb}%
     \NAT@date}}
  {\@citea\NAT@nmfmt{\NAT@nm}%
   \NAT@spacechar\NAT@@open\if*#1*\else#1\NAT@spacechar\fi\NAT@hyper@{\NAT@date}}
  {}{}

\makeatother

\begin{document}
\def\spacingset#1{\renewcommand{\baselinestretch}%
{#1}\small\normalsize} \spacingset{1}
\begin{flushleft}
{\Large{\textbf{A Geometric Condition for Uniqueness of Fréchet Means of Persistence Diagrams}}}
\newline
\\
Yueqi Cao$^{1}$ and Anthea Monod$^{1,\dagger}$
\\
\bigskip
\bf{1} Department of Mathematics, Imperial College London, UK
\\
\bigskip
$\dagger$ Corresponding e-mail: y.cao21@imperial.ac.uk
\end{flushleft}


\section*{Abstract}

The Fréchet mean is an important statistical summary and measure of centrality of data; it has been defined and studied for persistent homology captured by  persistence diagrams. However, the complicated geometry of the space of persistence diagrams implies that the Fréchet mean for a given set of persistence diagrams is not necessarily unique, which prohibits theoretical guarantees for empirical means with respect to population means. In this paper, we derive a variance expression for a set of persistence diagrams exhibiting a multi-matching between the persistence points known as a grouping.  Moreover, we propose a condition for groupings, which we refer to as flatness; we prove that sets of persistence diagrams that exhibit flat groupings give rise to unique Fréchet means. We derive a finite sample convergence result for general groupings, which results in convergence for Fréchet means if the groupings are flat. We then interpret flat groupings in a recently-proposed general framework of Fréchet means in Alexandrov geometry. Finally, we show that for manifold-valued data, the persistence diagrams can be truncated to construct flat groupings.

\paragraph{Keywords:} Alexandrov geometry; Fréchet means; persistent homology; nonnegative curvature.


\section{Introduction}
\label{sec:intro}

Persistent homology is an important methodology from topological data analysis which has gained rapid interest in recent decades and by now has been widely implemented in many applications across diverse scientific domains.  Given that its primary purpose is to summarize topological and geometric aspects of data---specifically, it captures the ``shape'' and ``size'' of a given dataset in the form of a \emph{persistence diagram}---studying statistical aspects of persistent homology is a central question in topological data analysis.  In statistics, the Fréchet mean is an extension of the usual arithmetic mean to general metric spaces, which is known to exist in the space of persistence diagrams and has been previously defined and studied for sets of persistence diagrams \cite{mileyko2011probability,turner2013means}.  In general, Fréchet means need not be unique and this is often the case for persistence diagrams, which significantly complicates important statistical questions, such as those concerning convergence.  This paper studies the metric geometry of the space of persistence diagrams to give conditions for unique Fréchet means and derive corresponding convergence results.

\paragraph{Related Work.}

Results on Fréchet means for sets of persistence diagrams have been established including convergence assuming uniqueness of the Fréchet mean \citep{turner2014frechet}.  Due to the complicated geometry of the space of persistence diagrams, it is not known when Fréchet means of sets of persistence diagrams are unique. The lack of a condition for uniqueness prohibits a comprehensive convergence analysis for empirical Fréchet means of persistence diagrams computed from real datasets \citep{cao2022approximating}.  We are thus restricted to only making descriptive and exploratory observations with Fréchet means computed from sampled data and cannot inference on the general behavior of the data generating process and the general unseen population with any theoretical guarantees.

Recently, a general theory for empirical Fréchet means in Alexandrov spaces with curvature bounds from either below or above has been proposed \citep{le2022fast}; this setting is in fact the geometric characterization of the space of persistence diagrams equipped with the 2-Wasserstein distance \citep{turner2014frechet}. The developed theory has been successfully applied to many cases, for example, barycenters of Gaussian distributions \citep{altschuler2021averaging}, template deformation models \citep{yoav}, and metric measure Laplacians \citep{mordant2022statistical}.\\

\paragraph{Contributions.}

Our main results are summarized as follows:

\begin{itemize}
\item We propose a geometric condition on sets of persistence diagrams that guarantees uniqueness of Fréchet means.  In particular, we consider a multi-matching representation between persistence points known as a {\em grouping} \citep{munch2015probabilistic}.

\item We derive a variance expression for general groupings.  We then propose a geometric condition on groupings, which we refer to as {\em flatness}, and show that flat groupings give rise to unique Fréchet means.

\item We then derive a finite sample convergence rate for groupings where the population probability measure is supported on a finite set. Our convergence result applies to Fréchet means of persistence diagrams if there exists a flat grouping.

\item We provide a position study of the space of persistence diagrams in the the general framework of Alexandrov spaces with curvature bounded from above or below \cite{le2022fast}, discussing the limitations of the general framework in the case of persistence diagram space.  We then provide explicit computations of particular geometric quantities of flat groupings and reconstruct results to constitute a better theoretical understanding of the Alexandrov geometry of Fréchet means of sets of persistence diagrams.

\item We discuss and numerically demonstrate the role of flat groupings in the problem of approximating the true persistence diagram of manifold-valued data using the Fréchet mean.  Specifically, we show that flat groupings may be constructed from persistence diagrams where a strip along the diagonal is truncated, so then we get uniqueness of Fréchet means of truncated persistence diagrams which can then be used to approximate persistent homology.

\end{itemize}

The remainder of this paper is organized as follows.  In Section \ref{sec:review}, we provide background and details on persistent homology and metric geometry, and in particular, the metric geometry of persistence diagrams and the space of persistence diagrams.  In Section \ref{sec:group}, we recall definitions of a grouping and a Fréchet mean, which are our specific objects of interest in this paper.  Here, we also present our contributions of an expression for the variance of general groupings, our proposed notion of flatness of groupings, and prove uniqueness of Fréchet means for sets of persistence diagrams for which there exist flat groupings.  We also derive finite sample convergence rates for general and flat groupings.  Section \ref{sec:conv} presents the flat grouping under the framework of Fréchet mean in Alexandrov spaces.  Section \ref{sec:truncatedPD} provides a practical demonstration of the role of flat groupings in a problem of approximating the persistent homology of manifold-valued data.  We close in Section \ref{sec:end} with a discussion of our findings and some ideas for future research based on our contributions in this paper.


\section{Background: Persistent Homology, Metric Geometry,\\
and Metric Geometry of Persistent Homology}
\label{sec:review}

In this section, we provide background and details on our setting and objects of study: persistent homology, which gives rise to persistence diagrams, and the space of all persistence diagrams.  We also review some concepts from metric geometry that will be essential for our study and construction of our results.

\subsection{Persistent Homology}

The standard pipeline of persistent homology begins with a filtration, which is a nested sequence of topological spaces: $\mathcal{M}_0 \subseteq \mathcal{M}_1 \subseteq \cdots \subseteq \mathcal{M}_n=\mathcal{M}$. By applying the homology functor $H(\cdot)$ with coefficients in a field, we have the sequence of homology vector spaces
$H(\mathcal{M}_0)\to H(\mathcal{M}_1)\to\cdots\to H(\mathcal{M}_n)$.
The collection of vector spaces $H(\mathcal{M}_i)$, together with vector space homomorphisms $H(\mathcal{M}_i)\to H(\mathcal{M}_j),i<j$, is called a \emph{persistence module}. When each $H(\mathcal{M}_i)$ is finite dimensional, a persistence module can be decomposed into a direct sum of irreducible summands called \emph{interval modules}, which correspond to birth and death times of homology classes \citep{chazal2016structure}. The collection of birth--death intervals $[\epsilon_i,\epsilon_j)$ is called a \emph{barcode} and it represents the \emph{persistent homology} of the filtration of $\mathcal{M}$. Each interval can also be identified as the coordinate of a point in the plane $\mathbb{R}^2$. In this way we have an alternate representation known as a \emph{persistence diagram}. 
For a detailed introduction of persistence homology, see e.g., \cite{edelsbrunner2008persistent,edelsbrunner2010computational}.

\begin{definition}
\label{def:PD}
A {\em persistence diagram} $D$ is a locally finite multiset of points in the half-plane $\Omega = \{(x,y)\in\mathbb{R}^2 \mid x<y\}$ together with points on the diagonal $\partial\Omega=\{(x,x)\in\mathbb{R}^2\}$ counted with infinite multiplicity. Points in $\Omega$ are called {\em off-diagonal points}. The persistence diagram with no off-diagonal points is called the {\em empty persistence diagram}, denoted by $D_\emptyset$.
\end{definition}

The geometry and statistical properties of the space of persistence diagrams are the main focus of this paper. 

\subsection{Metric Geometry}\label{sec:metric-geometry}

We now outline essential concepts from metric geometry and refer to \cite{burago2001course} for a comprehensive and detailed discussion. 

Let $(\mathcal{M},d)$ be an arbitrary metric space. For any two points $x,y\in \mathcal{M}$, a \emph{geodesic} connecting $x$ and $y$ is a continuous curve $\gamma:[a,b]\to\mathcal{M}$ such that for any $a\leq s\leq t\leq b$,
\begin{equation*}
    d(\gamma(s),\gamma(t)) = \frac{t-s}{b-a}d(x,y).
\end{equation*}
$(\mathcal{M},d)$ is called a \emph{geodesic space} if any two points can be joined by a geodesic. A geodesic space is an \emph{Alexandrov space with nonnegative curvature} if for every triangle $\{x_0,x_1,y\}\subseteq \mathcal{M}$ and a geodesic $\gamma:[0,1]\to \mathcal{M}$ connecting $x_0$ and $x_1$ there exists an isometric triangle $\{\tilde{x}_0,\tilde{x}_1,\tilde{y}\}$ in $\mathbb{R}^2$ such that 
$
    d(y,\gamma(t))\ge \|\tilde{y}-\tilde{\gamma}(t)\|,
$
where $\tilde{\gamma}(t) = t\tilde{x}_1+(1-t)\tilde{x}_0$ is the line segment joining $\tilde{x}_0$ and $\tilde{x}_1$ in the Euclidean plane. 

Given $z\in \mathcal{M}$, let $\Gamma_z$ be the set of all geodesics emanating from $z$. For any two geodesics $\gamma_0,\gamma_1\in\Gamma_z$, the \emph{Alexandrov angle} $\angle_z(\gamma_0,\gamma_1)$ is defined by
\begin{equation*}
    \angle_z(\gamma_0,\gamma_1) = \lim_{s,t\to 0}\cos^{-1}\left(\frac{d^2(z,\gamma_0(t))+d^2(z,\gamma_1(s))-d^2(\gamma_0(t),\gamma_1(s))}{2d(z,\gamma_0(t))d(z,\, \gamma_1(s))}\right).
\end{equation*}
If $(\mathcal{M},d)$ is an Alexandrov space with nonnegative curvature then $\angle_z:\Gamma_z\times \Gamma_z\to[0,\pi]$ is well-defined and a pseudo-metric on $\Gamma_z$. Therefore $\angle_z$ defines a metric on the quotient space $\Gamma_z/\sim$ where $\gamma_0\sim\gamma_1$ if and only if $\angle_z(\gamma_0,\gamma_1)=0$. 

The completion $(\widehat{\Gamma}_z,\angle_z)$ of $(\Gamma_z/\sim,\angle_z)$ is called the \emph{space of directions}. Let $v_\gamma$ denote the direction of $\gamma$ at $z$, i.e., the equivalence class of $\gamma$ in $\widehat{\Gamma}_z$. The \emph{tangent cone} $T_z\mathcal{M}$ is defined as $\widehat{\Gamma}_z\times \mathbb{R}_+/\sim$ where $(v_\gamma,t)\sim (v_\eta,s)$ if and only if $t=s=0$ or $(v_\gamma,t)=(v_\eta,s)$. Let $[v_\gamma,t],\, [v_\eta,s]\in T_z\mathcal{M}$ be two tangent vectors, then define 
\begin{equation}
\label{eq:cone_metric}
    \mathrm{C}_z([v_\gamma,t],[v_\eta,s]) = \sqrt{s^2+t^2-2st\cos\angle_z(v_\gamma,v_\eta)}.
\end{equation}
$\mathrm{C}_z$ is a metric on $T_z\mathcal{M}$ called the \emph{cone metric}.

For a geodesic space, the \emph{logrithmic map (log map) at $z$} $\log_z:\mathcal{M}\to T_z\mathcal{M}$ assigns $x$ to $[v_\gamma,\, d(z,x)]$ where $\gamma$ is a geodesic from $z$ to $x$. The log map is a multimap since there can be different geodesics from $z$ to $x$. By selecting an arbitrary direction for every $x$ the log map is a well-defined map, and moreover $\log_z$ can be chosen to be measurable with respect to the Borel algebra of $T_z\mathcal{M}$ \citep{le2022fast}.  

\subsection{Metric Geometry of Persistence Diagram Space}

The collection of all persistence diagrams may be viewed as a space; in particular, it is a metric space and hence its metric geometry may be studied \citep{che2024metric}.  Although there exist various possible metrics on the space of persistence diagrams, we focus on the following.

\begin{definition}
For any two persistence diagrams $D_1$ and $D_2$, define the {\em 2-Wasserstein distance} by
$$
\mathrm{W}_2(D_1,D_2) = \inf_{\phi}\left(\sum_{x\in D_1}\|x-\phi(x)\|^2\right)^{\frac{1}{2}}
$$
where $\phi$ ranges over all bijections between $D_1$ and $D_2$, and $\|\cdot\|$ denotes the 2-norm on $\mathbb{R}^2$. The {\em total persistence} of a persistence diagram $D$ is defined as $\mathrm{W}_2(D,D_\emptyset)$. Let $\mathcal{S}_2$ be the set of all persistence diagrams with finite total persistence. 
\end{definition}

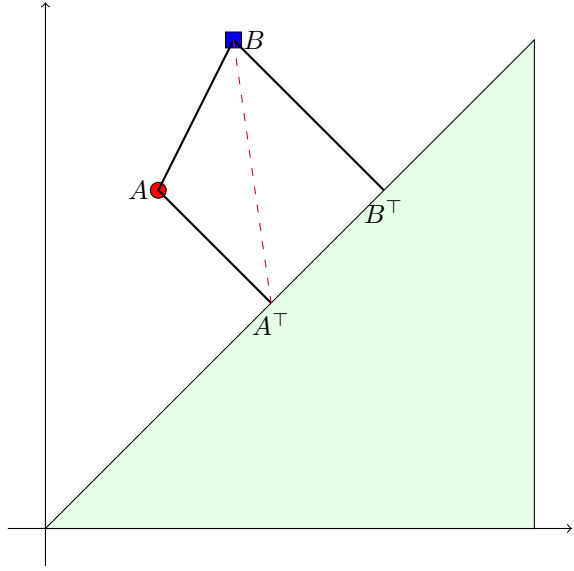
\begin{figure}
    \centering
    \begin{tikzpicture}
    \draw[->] (-0.5,0)--(7,0);
    \draw[->] (0,-0.5)--(0,7);

    \filldraw[fill=green!10] (0,0) -- (6.5,0) -- (6.5,6.5) -- cycle;

    \filldraw[fill=red] (1.5,4.5) circle (3pt);

    \draw[black, thick] (1.5,4.5)node[left]{$A$} -- (3,3)node[below]{$A^\top$};

    \filldraw[fill=blue] ([xshift=-3pt,yshift=-3pt]2.5,6.5) rectangle ++(6pt,6pt);

    \draw[black, thick] (2.5,6.5)node[right]{$B$} -- (4.5,4.5)node[below]{$B^\top$};

    \draw[black, thick] (1.5,4.5) -- (2.5,6.5); 

    \draw[purple, dash pattern=on 3pt off 4pt] (2.5,6.5) -- (3,3); 

    \end{tikzpicture}
    \caption{Curvature is determined by the boundary. Consider three persistence diagrams: diagram $D_A$ with a single off-diagonal point $A$, diagram $D_B$ with a single off-diagonal point $B$, and the empty diagram $D_\emptyset$ with no off-diagonal point. The three edges of triangle $\triangle D_\emptyset D_AD_B$ are plotted with solid lines. For the comparison triangle, given $\mathrm{W}_2(D_A,D_\emptyset)=\|AA^\top\|$, $\mathrm{W}_2(D_A,D_B)=\|AB\|$, and $\angle D_\emptyset D_AD_B=\angle A^\top AB$, the length of the third edge is $\|A^\top B\|$. We see that $\|A^\top B\|>\|BB^\top\|=\mathrm{W}_2(D_B,D_\emptyset)$, indicating nonnegative Alexandrov curvature.}
    \label{fig:curveness}
\end{figure}

Under the 2-Wasserstein distance, we now discuss several metric geometric characteristics of the space of persistence diagrams.  We have that $(\mathcal{S}_2,\mathrm{W}_2)$ is an Alexandrov space with nonnegative curvature \citep{turner2014frechet}; the curvature behavior is largely determined by the boundary, see Figure \ref{fig:curveness}.  Moreover, we have the following characterization of geodesics between persistence diagrams: Let $D_1$ and $D_2$ be two persistence diagrams with finite total persistence and $\phi:D_1\to D_2$ be an optimal matching. Then the geodesic $\gamma:[0,1]\to\mathcal{S}_2$ joining $D_1$ to $D_2$ is such that $\gamma(t)$ is in fact a persistence diagram with points of the form $(1-t)x+t\phi(x)$ where $x$ ranges all points from $D_1$.

We also have the following characterization concerning tangent vectors.  Let $D\in\mathcal{S}_2$ be a persistence diagram.  A tangent vector in the tangent cone $T_D\mathcal{S}_2$ is a collection of vectors $\{v_j\}_{j\in J}\subseteq \mathbb{R}^2$ such that  $\sum_{j\in J}\|v_j\|^2<\infty$; and if $I$ is the index set of off-diagonal points in $D$, then $I\subseteq J$ and $\frac{v_j}{\|v_j\|}=\frac{1}{\sqrt{2}}(-1,1)$ for $j\in J\backslash I$.

If a persistence diagram $D$ consists of finitely many off-diagonal points---that is, if $I$ is a finite set---then the tangent vectors at $D$ can be realized as tangent vectors of geodesics originating from $D$. In general, there may exist tangent vectors with no corresponding geodesics, as we show in the following example.

\begin{example}
Consider the persistence diagram $D=\{x_n=\big(0,\frac{1}{n^2}\big),\, n\in\mathbb{N}\}$. At each $x_n$, assign the vector $v_n=\big(\frac{1}{n},-\frac{1}{n}\big)$. We claim that the collection $V=\{v_n,n\in\mathbb{N}\}$ is an element in $T_D\mathcal{S}_2$. In fact, let $V_N=\{\tilde{v}_n=v_n,1\leq n\leq N\}\cup\{\tilde{v}_n=0,n>N\}$. The geodesic $\gamma_N(t)=\{x_n+t\tilde{v}_n,n\in\mathbb{N}\}$ is well-defined for $t\in [0,\frac{1}{2N}]$. For $1\leq N\leq M$,
\begin{align*}
    \cos\angle_D(V_N,V_M) &=\lim_{s,t\to 0}\frac{\mathrm{W}_2^2(\gamma_N(t),D)+\mathrm{W}_2^2(\gamma_M(s),D)-\mathrm{W}_2^2(\gamma_N(t),\gamma_M(s))}{2\mathrm{W}_2(\gamma_N(t),D)\mathrm{W}_2(\gamma_M(s),D)}\\
    &\geq \lim_{s,t\to 0}\frac{\dsum_{n=1}^N\frac{2t^2}{n^2}+\dsum_{n=1}^M\frac{2s^2}{n^2}-\dsum_{n=1}^N\frac{2(t-s)^2}{n^2}-\dsum_{n=N+1}^M\frac{2s^2}{n^2}}{2\sqrt{\dsum_{n=1}^N\frac{2t^2}{n^2}}\sqrt{\dsum_{n=1}^M\frac{2s^2}{n^2}}}= \frac{\sqrt{\dsum_{n=1}^N\frac{1}{n^2}}}{\sqrt{\dsum_{n=1}^M\frac{1}{n^2}}}
\end{align*}
Hence under the cone metric,
\begin{align*}
    \mathrm{C}_D^2(V_N,V_M) &= \|V_N\|^2+\|V_M\|^2-2\|V_N\|\|V_M\|\cos\angle_D(V_N,V_M)\\
    &\leq 2\dsum_{n=1}^N\frac{1}{n^2}+2\dsum_{n=1}^M\frac{1}{n^2}-2\sqrt{2\dsum_{n=1}^N\frac{1}{n^2}}\cdot\sqrt{2\dsum_{n=1}^M\frac{1}{n^2}}\cdot\frac{\sqrt{\dsum_{n=1}^N\frac{1}{n^2}}}{\sqrt{\dsum_{n=1}^M\frac{1}{n^2}}}\\
    &=2\dsum_{n=N+1}^M\frac{1}{n^2}\to 0,\quad  N,M\to\infty
\end{align*}
Therefore, $\{V_N\}_{N\in\mathbb{N}}$ is a Cauchy sequence in $T_D\mathcal{S}_2$ and converges to $V$. However, $V$ is not a tangent vector of any geodesic emanating from $D$ as for any fixed $t$, $x_n+tv_n\notin \Omega$ when $n$ is sufficiently large.  
\end{example}

\section{Groupings of Persistence Diagrams and their Fréchet Means}\label{sec:group}

Let $(\mathcal{M},d)$ be a metric space and $\mu$ be a (Borel) probability measure on $\mathcal{M}$. The \emph{Fréchet function} is defined by
\begin{equation*}
    F(x) = \int_{\mathcal{M}}d^2(x,y)\diff{\mu(y)}.
\end{equation*}
If $F(x)$ is finite for some (hence, every) $x$, the probability measure $\mu$ is said to have finite second moment. The quantity $\mathbb{V}=\inf_{x\in\mathcal{M}}F(x)$ is the {\em variance} of $\mu$. The set of points achieving the variance is the {\em Fréchet mean or expectation}. 

Fréchet means for sets of persistence diagrams exist, given that a probability measure on $(\mathcal{S}_2,\mathrm{W}_2)$ has finite second moment and compact support \citep{mileyko2011probability,turner2014frechet}. However, Fréchet means are not necessarily unique due to the nonnegative curvature of $(\mathcal{S}_2,\mathrm{W}_2)$. The lack of unique Fréchet means is problematic in many practical applications as well as theoretical settings---for example, averaging time-varying persistence diagrams \citep{munch2015probabilistic} and establishing convergence of empirical Fréchet mean of persistence diagrams \citep{cao2022approximating}---however, approximations for Fréchet means are computable.  

Let $D_1,\ldots,D_L$ be a finite set of persistence diagrams with finite off-diagonal points, and $\mu=\frac{1}{L}\sum_{i=1}^L\delta_{D_i}$ be a discrete probability measure. The Fréchet function for this set of persistence diagrams is 
\begin{equation}\label{eq:frechet-function}
    F(D) = \frac{1}{L}\dsum_{i=1}^L\mathrm{W}_2^2(D,D_i).
\end{equation}
\cite{turner2014frechet} proposed a greedy algorithm to compute local minima of the Fréchet function \eqref{eq:frechet-function}.  Other work by \cite {lacombe2018large} in more general contexts has also given rise to alternative algorithms to compute Fréchet means for persistence diagrams.  \cite{munch2015probabilistic} introduced probabilistic Fréchet means to average time-varying persistence diagrams. We now recall and rephrase some definitions and results from some of this prior work which will be useful for our study.

\begin{definition}
Let $D_1,\ldots,D_L$ be a finite set of $L$ persistence diagrams, each with $k_1,\ldots,k_L$ off-diagonal points. Let $K=k_1+\cdots+k_L$ be the total number of off-diagonal points. A {\em grouping} $G$ is a $K\times L$ formal matrix whose elements are off-diagonal points from $D_1,\ldots,D_L$ and copies of the diagonal $\partial\Omega$. Let $G^j$ be the $j$th column of $G$. It consists of all $k_j$ off-diagonal points of $D_j$ and $K-k_j$ copies of the diagonal $\partial\Omega$. Each row of $G$ is called a {\em selection}. A \emph{trivial selection} is a row with all $\partial\Omega$ entries.
\end{definition}

Intuitively, a grouping is a matching of points between persistence diagrams.  For the special case where $L=2$, a grouping is equivalent to a bijective matching between two persistence diagrams, with each selection representing the one-to-one correspondence between points. In general cases, a grouping is thus a representation of multi-matching, i.e., any two columns of the grouping induce a bijective matching between corresponding persistence diagrams.

\begin{remark}
    In the framework of optimal transport theory, a grouping can be understood as a collection of transport plans solving the multi-marginal transport problem: Let $\mathcal{M}_1,\ldots,\mathcal{M}_n$ be compact metric spaces equipped with Borel probability measures $\mu_1,\ldots,\mu_n$. Let $\mathcal{Z}$ be the target compact metric space. Suppose for each $i\in\{1,\ldots,n\}$ there is a continuous cost function $c_i:\mathcal{M}_i\times \mathcal{Z}\to\mathbb{R}$. The multi-marginal optimal transport problem seeks to find a probability measure $\nu$ on $\mathcal{Z}$ minimizing the following objective function:
$$
J(\nu) = \sum_{i=1}^n \mathrm{W}_{c_i}(\mu_i,\nu) = \sum_{i=1}^n \inf_{\pi_i\in\Pi(\mu_i,\nu)}\left\{\int_{\mathcal{M}_i\times \mathcal{Z}}c_i(x,z)\diff{\pi_i}(x,z)\right\}.
$$
It was shown in \cite{carlier2015numerical} that it suffices to find probability measures $\pi_i^*$ on $\mathcal{M}_i\times \mathcal{Z}$ such that
$$
(\pi_1^*,\ldots,\pi_n^*) =\mathop{\arg\min}_{\substack{(\pi_1,\ldots,\pi_n) \\ \in \Pi(\mu_1,\cdots,\mu_n)}}\sum_{i=1}^n \int_{\mathcal{M}_i\times \mathcal{Z}}c_i(x,z)\diff{\pi_i}(x,z)
$$
where $\Pi(\mu_1,\ldots,\mu_n)$ consists of tuples $(\pi_1,\ldots,\pi_n)$ such that for any Borel set $A_i\subseteq \mathcal{M}_i$ and $B\subseteq \mathcal{Z}$,
$$
\pi_i(A_i\times \mathcal{Z})=\mu_i(A_i),\quad \pi_1(\mathcal{M}_1\times B)=\cdots=\pi_n(\mathcal{M}_n\times B).
$$
Let $G$ be a grouping of the persistence diagrams $D_1,\ldots, D_L$. Each column of $G$ defines a discrete probability measure $\mu_i$ on $\mathbb{R}^2\cup\{\partial\Omega\}$. The mean persistence diagram $\mathrm{mean}(G)$ induces the target probability measure $\nu$. The rows of $G$ give transport plans $\pi_i$ from $\mu_i$ to $\nu$. Thus, the formulation of multi-marginal optimal transport coincides with a grouping for a finite set of persistence diagrams. 
\end{remark} 

For any $x\in\Omega$, let $x^\top$ be the projection to the diagonal, and $x^\bot=x-x^\top$. Set $\|x-\partial\Omega\|=\|x^\bot\|$ and $\|\partial\Omega-\partial\Omega\|=0$.  Let $Q=\{x_1,\ldots,x_L\}$ be a multiset of off-diagonal points and copies of the diagonal. If $Q\subseteq \Omega$ consists of off-diagonal points only, the mean point $\bar{Q}$ is the usual algorithmic mean $\bar{Q}=\frac{1}{L}\sum_{i=1}^Lx_i$. If $Q=\{x_1,\ldots,x_{s}\}\cup\{\partial\Omega,\ldots,\partial\Omega\}$ consists of $(L-s)$ copies of the diagonal, set $Q_{o} = \{x_1,\ldots,x_{s}\}$, and the mean point is then given by 
\begin{equation}\label{eq:mean-with-diagonal}
    \bar{Q}=\frac{s\bar{Q}_{o}+(L-s)(\bar{Q}_{o})^\top}{L}.
\end{equation}
If $Q=\{\partial\Omega,\ldots,\partial\Omega\}$, then $\bar{Q}=\partial\Omega$.

\begin{definition}
Let $\{D_1,\ldots,D_L\}$ be a set of persistence diagrams and $G$ be a grouping of size $K\times L$. The {\em mean persistence diagram} $\mathrm{mean}(G)$ is the diagram where each off-diagonal point is given by $\bar{G}_i$ for each nontrivial selection $G_i$. The {\em variance} of $G$ is defined as 
    \begin{equation}
        \mathbb{V}(G) = \frac{1}{L}\sum_{j=1}^L\sum_{i=1}^K\|G_i^j-\bar{G}_i\|^2.
    \end{equation}
\end{definition} 

The following theorem establishes the relation between Fréchet means and groupings of persistence diagrams.

\begin{theorem}{\cite[Theorem 3.3]{turner2014frechet}}\label{thm:turner}
Given a finite set of persistence diagrams $D_1,\ldots,D_L$, if $D_\star$ is a Fréchet mean then $D_\star=\mathrm{mean}(G_\star)$ for some grouping $G_\star$, and the optimal matching between $D_\star$ and each $D_i,i=1,\ldots,L$ is induced by $G_\star$. 
\end{theorem}

This result allows us to consider the Fréchet variance as the minimal variance of groupings, 
\begin{equation*}
    \sigma^2 = \min_{D} \frac{1}{L}\sum_{i=1}^L \mathrm{W}_2^2(D,D_i)\\
    = \min_{G} \mathbb{V}(G).
\end{equation*}
{\em Optimal groupings} are groupings that achieve the Fréchet variance. For a general grouping we derive the following variance expression.

\begin{figure}
    \centering
    \begin{tikzpicture}[scale=0.7]

    \draw[dashed] (-3,-1)--(7.5,-1);
    \draw[->] (-2,-1.5)--(-2,9);

    \filldraw[fill=green!10] (4,-1) -- (7,-1) -- (7,2) -- cycle;
    \draw[dashed] (2,-1) -- (7,4);

    \draw[dashed] (2,2) circle (1.5);
    \filldraw[fill=red] (1,2.5) circle (3pt);
    \filldraw[fill=blue] (3,1.5) circle (3pt);
    \filldraw[fill=cyan] (2,1) circle (3pt);
    \node[star, star points=5, star point ratio=2,fill=black, scale=0.5] at (2,1.67) {};
    \draw[densely dotted] (1,2.5) -- (2,1.67);
    \draw[densely dotted] (3,1.5) -- (2,1.67);
    \draw[densely dotted] (2,1) -- (2,1.67);

    \draw[dashed] (4.5,6.5) circle (1.5);
    \filldraw[fill=red] (4,6) circle (3pt);
    \filldraw[fill=blue] (5,7.5) circle (3pt);
    \filldraw[fill=cyan] (5.4,6) circle (3pt);
    \node[star, star points=5, star point ratio=2,fill=black, scale=0.5] at (4.8,6.5) {};
    \draw[densely dotted] (4,6) -- (4.8,6.5);
    \draw[densely dotted] (5,7.5) -- (4.8,6.5);
    \draw[densely dotted] (5.4,6) -- (4.8,6.5);

    \draw[dashed] (0,6) circle (1.5);
    \filldraw[fill=red] (0.5,6.5) circle (3pt);
    \filldraw[fill=blue] (-0.5,5.5) circle (3pt);
    \filldraw[fill=cyan] (0,7) circle (3pt);
    \node[star, star points=5, star point ratio=2,fill=black, scale=0.5] at (0,6.33) {};
    \draw[densely dotted] (0.5,6.5) -- (0,6.33);
    \draw[densely dotted] (-0.5,5.5) -- (0,6.33);
    \draw[densely dotted] (0,7) -- (0,6.33);

    \end{tikzpicture}
    \caption{An example of flat groupings. The off-diagonal points of $D_{\mathrm{red}},D_{\mathrm{blue}},D_{\mathrm{cyan}}$ are distributed as three clusters over the half-plane $\Omega$. Every dashed circle indicates a selection of the grouping. The Fréchet mean is given by $D_{\mathrm{black}}$.}
    \label{fig:well-separated}
\end{figure}
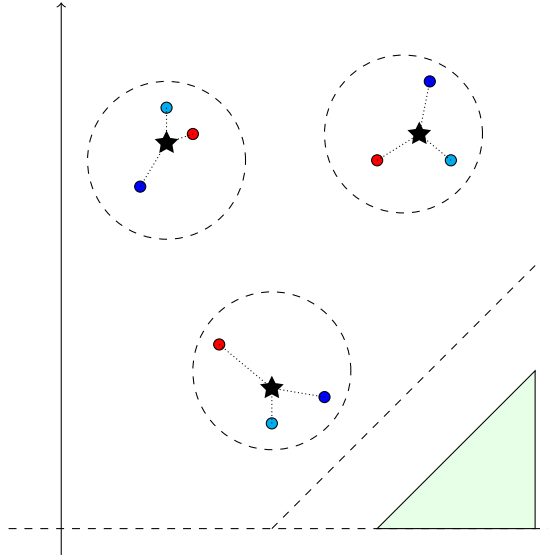

\begin{theorem}\label{thm:var-formula}
Let $\{D_1,\ldots,D_L\}$ be a set of persistence diagrams, and $G$ be a grouping of size $K\times L$. Let $s_i$ be the number of off-diagonal points in the $i$th row of $G$. The variance of $G$ is
\begin{equation}\label{eq:grouping-variance}
    \mathbb{V}(G) = \frac{1}{L^2}\sum_{i=1}^K\sum_{1\le w<\ell\le L}\|G_i^w-G_i^\ell\|^2+\sum_{i=1}^K\frac{L-s_i}{L^2s_i}\left(\sum_{1\le w<\ell\le s_i}\|(G_i^{j_w})^\top-(G_i^{j_\ell})^\top\|^2\right),
\end{equation}
where $G_i^{j_\ell},\ell=1,\ldots,s_i$ ranges over all off-diagonal points in the $i$th row of $G$. If $s_i=0$, the summand is taken to be $0$.
\end{theorem}

\begin{proof}
Let $G_i$ be the $i$th row with $s_i>0$ off-diagonal points $G_i^{j_1},\ldots,G_i^{j_{s_i}}$. Then the variance for $G_i$ is
\begin{equation}\label{eq:var-decomp}
    \mathbb{V}(G_i) = \frac{1}{L}\sum_{j=1}^L\|G_i^j-\bar{G}_i\|^2 
    %
\end{equation}
Following \eqref{eq:mean-with-diagonal}, we have
\begin{equation*}
    \bar{G}_i^\top=\frac{1}{s_i}\sum_{\ell=1}^{s_i}(G_i^{j_\ell})^\top,\quad\bar{G}_i^\bot=\frac{1}{L}\sum_{\ell=1}^{s_i}(G_i^{j_\ell})^\bot
\end{equation*}
Recall that $\|G_i^j-\bar{G}_i\|=\|\bar{G}_i^\bot\|$ if $G_i^j=\partial\Omega$ is the diagonal. For notational convenience, we denote $\partial\Omega^\top := \bar{G}_i^\top$ and $\partial\Omega^\bot = 0$; then \eqref{eq:var-decomp} decomposes as
\begin{equation}\label{eq:var-decomp2}
    \mathbb{V}(G_i) = \frac{1}{L}\bigg(\sum_{\ell=1}^{s_i}\|(G_i^{j_\ell})^\top-\bar{G}_i^\top\|^2+\sum_{j=1}^{L}\|(G_i^j)^\bot-\bar{G}_i^\bot\|^2\bigg).
\end{equation}

Observe that given $a_1,\ldots,a_n\in\mathbb{R}$ and $\bar{a} = \frac{1}{n}\sum_{\ell=1}^na_\ell$, we have

\begin{equation*}
    \begin{aligned}
    \sum_{w,\ell}(a_w-a_\ell)^2 &= \sum_{w,\ell}((a_w-\bar{a})-(a_\ell-\bar{a}))^2\\
    &= n\sum_{w}(a_w-\bar{a})^2-2\sum_{w,\ell}(a_w-\bar{a})(a_\ell-\bar{a})+n\sum_{\ell}(a_\ell-\bar{a})^2\\
    &= 2n\sum_{\ell=1}^n(a_\ell-\bar{a})^2
    \end{aligned}
\end{equation*}
which implies $\frac{1}{n}(a_i-\bar{a})^2 = \displaystyle\sum_{1\le w<\ell\le n}(a_w-a_\ell)^2$. Using this identity, \eqref{eq:var-decomp2} becomes
\begin{align*}
    \mathbb{V}(G_i) =& \frac{1}{L}\bigg(\frac{1}{s_i}\sum_{1\le w<\ell\le s_i}\|(G_i^{j_w})^\top-(G_i^{j_\ell})^\top\|^2+\frac{1}{L}\sum_{1\le w\le\ell\le L}\|(G_i^w)^\bot-(G_i^\ell)^\bot\|^2\bigg)\\
    =& \frac{1}{L}\bigg(\left(\frac{1}{L}+\frac{1}{s_i}-\frac{1}{L}\right)\sum_{1\le w<\ell\le s_i}\|(G_i^{j_w})^\top-(G_i^{j_\ell})^\top\|^2+\frac{1}{L}\sum_{1\le w\le\ell\le L}\|(G_i^w)^\bot-(G_i^\ell)^\bot\|^2\bigg)\\
    =& \frac{1}{L^2}\bigg(\sum_{1\le w<\ell\le s_i}\|(G_i^{j_w})^\top-(G_i^{j_\ell})^\top\|^2+\sum_{1\le w\le\ell\le L}\|(G_i^w)^\bot-(G_i^\ell)^\bot\|^2\bigg) +\frac{L-s_i}{L^2s_i}\sum_{1\le w<\ell\le s_i}\|(G_i^{j_w})^\top-(G_i^{j_\ell})^\top\|^2\\
    =& \frac{1}{L^2}\sum_{1\le w<\ell\le L}\|G_i^w-G_i^\ell\|^2 + \frac{L-s_i}{L^2s_i}\sum_{1\le w<\ell\le s_i}\|(G_i^{j_w})^\top-(G_i^{j_\ell})^\top\|^2,
\end{align*}
where the last step follows from the same reasoning as \eqref{eq:var-decomp2}. Finally, summing $\mathbb{V}(G_i)$ for all rows, we obtain \eqref{eq:grouping-variance}.
\end{proof}

Notice that if we disregard the diagonal $\partial\Omega$ and suppose $G$ is a grouping of points in the plane $\mathbb{R}^2$, then the variance of $G$ only consists of the first term in \eqref{eq:grouping-variance}. The diagonal contributes the second term in the variance expression.

The derived variance expression motivates the following definition.
\begin{definition}\label{def:flat_grouping}
A grouping $G$ is called \emph{flat} if there exists $\lambda>0$ such that
\begin{enumerate}[(i)]
    \item For each nontrivial selection $G_i$, the diameter is bounded above by $\lambda$, i.e., $\|G_i^w-G_i^\ell\|<\lambda$ for all $w,\ell=1,\ldots,L$;
    \item For two distinct selections $G_i,G_j$, the distance between $G_i$ and $G_j$ is bounded below by $\lambda$, i.e., $\|G_i^w-G_j^\ell\|>\lambda$ for all $w,\ell=1,\ldots,L$;
    \item Off-diagonal points are bounded away from the diagonal by $\lambda$, i.e., $\|G_i^w-\partial\Omega\|>\lambda$ for $G_i^w\neq \partial\Omega$.
\end{enumerate} 
\end{definition}

A visual example of flatness is illustrated in Figure \ref{fig:well-separated}.

Given this notion of flatness, we now have a condition that gives rise to unique Fréchet means of persistence diagrams.

\begin{theorem}\label{thm:unique-mean}
Let $\{D_1,\ldots,D_L\}$ be a set of persistence diagrams. If there exists a flat grouping $G_\star$ for $\{D_1,\ldots,D_L\}$, then $\mathrm{mean}(G_\star)$ is the unique Fréchet mean of $\{D_1,\ldots,D_L\}$.
\end{theorem}

\begin{proof}
Suppose $G_\star$ is a flat grouping. By conditions (i) and (iii), each nontrivial selection of $G$ does not contain the diagonal. Thus, the variance of $G_\star$ is 
\begin{equation*}
    \mathbb{V}(G_\star)=\frac{1}{L^2}\sum_{i=1}\sum_{1\le w<\ell\le L}\|(G_\star)_i^w-(G_\star)_i^\ell\|^2.
\end{equation*}
Let $G$ be any grouping. By \eqref{eq:grouping-variance}, we have 
\begin{equation*}
    \mathbb{V}(G)\ge \frac{1}{L^2}\sum_{1\le w<\ell\le L}\sum_{i=1}\|G_i^w-G_i^\ell\|^2.
\end{equation*}
Now, fix any two columns $w$ and $\ell$. Without loss of generality, we may assume $(G_\star)_i^w=G_i^w$ (otherwise, we may apply row permutation to achieve this form). By conditions (ii) and (iii),
\begin{equation*}
    \min\{\|G_i^w-G_i^\ell\|,\|G_i^w-\partial\Omega\|\}>\lambda> \|G_i^w-(G_\star)_i^\ell\|
\end{equation*}
for any $G_i^\ell\neq (G_\star)_i^\ell$. Therefore, $\mathbb{V}(G)>\mathbb{V}(G_\star)$ if $G\neq G_\star$. Thus $\mathrm{mean}(G_\star)$ is the unique Fréchet mean. 
\end{proof}

Notice that although the flat groupings are bounded away from the diagonal, proving uniqueness nevertheless requires the consideration of all possible groupings where the diagonal is involved. Thus Theorem \ref{thm:var-formula} is required here in the proof of Theorem \ref{thm:unique-mean}.

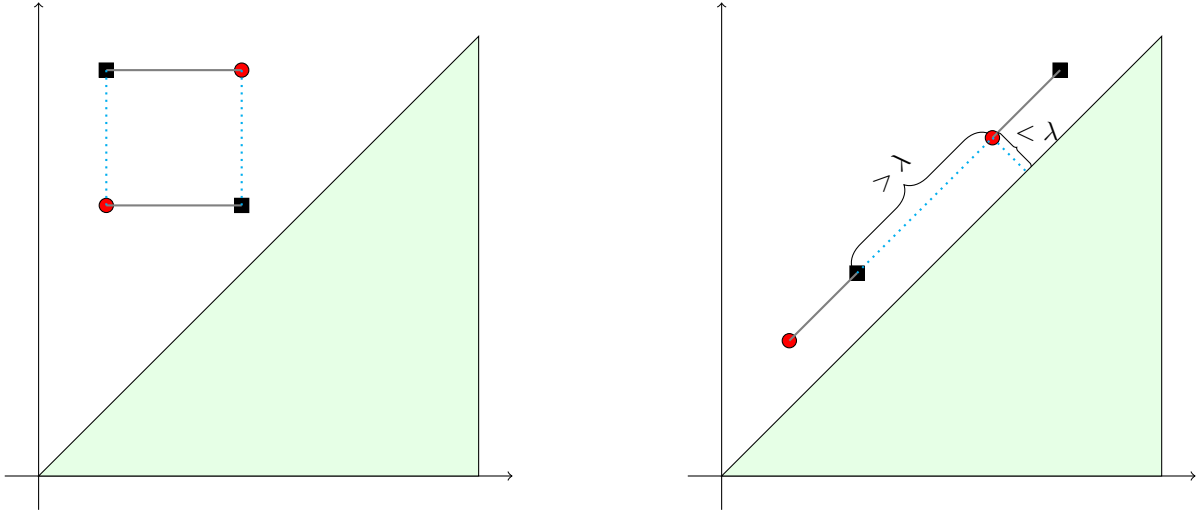
\begin{figure}
    \centering
    \begin{subfigure}{0.45\textwidth}
    \centering
    \begin{tikzpicture}[scale=0.75]
    \draw[->] (-0.5,0)--(7,0);
    \draw[->] (0,-0.5)--(0,7);

    \filldraw[fill=green!10] (0,0) -- (6.5,0) -- (6.5,6.5) -- cycle;

    \filldraw[fill=red] (1,4) circle (3pt);
    \filldraw[fill=red] (3,6) circle (3pt);

    \filldraw ([xshift=-3pt,yshift=-3pt]1,6) rectangle ++(6pt,6pt);
    \filldraw ([xshift=-3pt,yshift=-3pt]3,4) rectangle ++(6pt,6pt);

    \draw[gray, thick] (1,4) -- (3,4);
    \draw[gray, thick] (1,6) -- (3,6);

    \draw[cyan, dotted, thick] (1,4) -- (1,6);
    \draw[cyan, dotted, thick] (3,4) -- (3,6);
    \end{tikzpicture}
    \end{subfigure}
    \hfill
    \begin{subfigure}{0.45\textwidth}
    \centering
    \begin{tikzpicture}[scale=0.75]
    \newcommand{\annL}[6] [] {
    \coordinate (AE1) at ($(#2)!#4!90:(#3)$);
    \coordinate (BE1) at ($(#3)!#4!-90:(#2)$);
    \coordinate (AE2) at ($(#2)!#5!90:(#3)$);
    \coordinate (BE2) at ($(#3)!#5!-90:(#2)$);
    \draw[>=latex,red,<->]  (AE2) -- (BE2) node[midway,sloped,below,align=center] {#6};
    \draw[very thin,shorten >=1pt,shorten <=1pt] (#2) -- (AE1);
    \draw[very thin,shorten >=1pt,shorten <=1pt] (#3) -- (BE1);
    }
    \newcommand{\annC}[6] [] {
    \coordinate (AE1) at ($(#2)!#4!90:(#3)$);
    \coordinate (BE1) at ($(#3)!#4!-90:(#2)$);
    \draw [decorate,decoration={brace,amplitude=#5},xshift=0pt,yshift=0pt]
          (AE1) -- (BE1) node [black,midway,sloped,above,yshift=#5]  {#6} ;
    }

    \draw[->] (-0.5,0)--(7,0);
    \draw[->] (0,-0.5)--(0,7);

    \filldraw[fill=green!10] (0,0) -- (6.5,0) -- (6.5,6.5) -- cycle;

    \filldraw[fill=red] (1,2) circle (3pt);
    \filldraw[fill=red] (4,5) circle (3pt);

    \filldraw ([xshift=-3pt,yshift=-3pt]2,3) rectangle ++(6pt,6pt);
    \filldraw ([xshift=-3pt,yshift=-3pt]5,6) rectangle ++(6pt,6pt);

    \draw[gray, thick] (1,2) -- (2,3);
    \draw[gray, thick] (4,5) -- (5,6);

    \coordinate (O) at (1,2);
    \coordinate (A) at (2,3);
    \coordinate (B) at (4,5);
       \draw[cyan,dotted,thick] (O) -- (A);
    \annC{O}{A}{1mm}{3mm}{$\lambda-\epsilon$};
    \draw[cyan,dotted,thick] (A) -- (B);
    \annC{A}{B}{1mm}{3mm}{$\lambda+\epsilon$};
    \coordinate (C) at (4.5,4.5);
    \draw[cyan,dotted,thick] (B) -- (C);
    \annC{B}{C}{1mm}{0.5mm}{};
    \node[right] at (4.4,5.1) {$\epsilon$};

    \end{tikzpicture}
    \end{subfigure}
    \caption{Counterexamples violating the conditions of flat groupings (Definition \ref{def:flat_grouping}). On the left panel, a grouping for two persistence diagrams $D_{\mathrm{red}},D_{\mathrm{black}}$ is depicted by solid lines. Four off-diagonal points form the corners of a square, hence the grouping violates conditions (i) and (ii). The Fréchet mean is not unique as the grouping depicted by dotted lines gives another Fréchet mean. On the right panel, a grouping is given by the solid lines. Let $\epsilon>0$ be a small positive number. The grouping satisfies condition (i) in Definition \ref{def:flat_grouping} because each selection has diameter $\lambda-\epsilon$. The dotted line parallel to the diagonal shows the distance between two selections is $\lambda+\epsilon$. Thus the grouping satisfies condition (ii) of Definition \ref{def:flat_grouping}. The dotted line orthogonal to the diagonal shows that the distances of off-diagonal points to the diagonal are sufficiently smaller than $\lambda$. Thus the grouping violates condition (iii) of Definition \ref{def:flat_grouping}. In this case, the mean of the grouping is not a Fréchet mean, since the optimal grouping here will match all off-diagonal points with the diagonal and the Fréchet mean is given by the intermediate points between off-diagonal points and the diagonal.} 
    \label{fig:nonminimal}
\end{figure}

If there exists a flat grouping for $D_1,\ldots,D_L$, then the off-diagonal points are distributed as several clusters over the half-plane $\Omega$; see Figure \ref{fig:well-separated}. Though flat groupings are special, there are counterexamples if we drop any one of the three conditions; see Figure \ref{fig:nonminimal}. 

In terms of computation, we remark that flatness has implications on the computability of Fréchet means of persistence diagrams.  In \cite{turner2014frechet}, a greedy algorithm was introduced to compute the Fréchet mean of a set of $m$ persistence diagrams $D_1,\ldots,D_m$. If the algorithm initializes at some $D_i$, then the algorithm will converge globally in one step. If it starts at the midway point between two of the $m$ persistence diagrams, or if it starts with the mean persistence diagram of a random grouping, then it may not converge globally: Suppose $D_1=\{(1,10),(1,12)\}$ and $D_2=\{(2,10),(2,12)\}$ are two persistence diagrams. The flat grouping is given by
        $$
        \begin{bmatrix}
            (1,10) & (2,10)\\
            (1,12) & (2,12)
        \end{bmatrix}.
        $$
        However, if we start with grouping 
        $$
        \begin{bmatrix}
            (1,10) & (2,12)\\
            (1,12) & (2,10)
        \end{bmatrix}
        $$
        whose mean persistence diagram is $\{(1.5,11),\,(1.5,11)\}$, then the grouping remains the same in the following iterations and the algorithm will not output the true Fréchet mean.

We now turn to discussing convergence of Fréchet means for persistence diagrams. 
 For $\rho=\frac{1}{L}\sum_{i=1}^L\delta_{D_i}$ a discrete probability measure supported on $\{D_1,\ldots,D_L\}$ and $D'_1,\ldots,D'_B$ i.i.d.~samples drawn from $\rho$, \cite{turner2014frechet} proved that if the Fréchet mean for $\rho$ is unique, then with probability one, the empirical Fréchet mean converges to the population Fréchet mean under the Hausdorff distance, which is, to date, the only convergence result for Fréchet means of sets of persistence diagrams.  

A finite sample convergence rate is of practical importance in paving the way to establishing the Fréchet mean as a viable tool with theoretical guarantees in important computational settings, such as those discussed by \cite{cao2022approximating} where the goal is to find an appropriate representation to approximate the true persistent homology of a very large, yet finite, dataset. Using the language of groupings, we can derive a general and simple finite sample convergence rate, which reduces to the convergence for Fréchet means as a special case when the grouping is flat. 

\begin{theorem}
Let $\bm{\rho}=\frac{1}{L}\sum_{i=1}^L\delta_{D_i}$ be a discrete probability measure on $\mathcal{S}_2$, and $D_1',\ldots,D_B'$ be i.i.d.~samples from $\bm{\rho}$. Let $\mathbf{G}$ be a grouping for $D_1,\ldots,D_L$, and $D_\star=\mathrm{mean}(\mathbf{G})$ be the population mean persistence diagram. Let $G$ be the induced grouping for $D_1',\ldots,D_B'$, with each column $G^j$ being the corresponding column of $D_j'$ in $\mathbf{G}$, and $\bar{D} = \mathrm{mean}(G')$ be the empirical mean persistence diagram. We have
\begin{equation}\label{eq:rate-grouping}
    \mathbb{E}[\mathrm{W}_2^2(\bar{D},D_\star)]\le \frac{\sigma^2}{B}
\end{equation}
where $\sigma^2 = \mathbb{V}(\mathbf{G})$ is the variance of the grouping $\mathbf{G}$.
\end{theorem}
\begin{proof}
Since the rows of the grouping $\mathbf{G}$ give a natural bijection between $\bar{D}$ and $D_\star$, we have
$$
\mathbb{E}[\mathrm{W}_2^2(\bar{D},D_\star)]\le \sum_{i=1}^K\mathbb{E}[\|\bar{G}_i-\bar{\mathbf{G}}_i\|^2].
$$
It suffices to consider row-wise variances. Suppose there are $m_i$ off-diagonal points for the $i$th row $G_i$. Then
\begin{align}
    \mathbb{E}[\|\bar{G}_i-\bar{\mathbf{G}}_i\|^2] =&{} \mathbb{E}[\|\bar{G}_i^\top-\bar{\mathbf{G}}_i^\top\|^2]+ \mathbb{E}[[\|\bar{G}_i^\bot-\bar{\mathbf{G}}_i^\bot\|^2]] \nonumber\\
    =&{} \mathbb{E}\left[\bigg\|\frac{1}{m_i}\sum_{\ell=1}^{m_i}(G_i^{j_\ell})^\top-\bar{\mathbf{G}}_i^\top\bigg\|^2\right]\label{eq:var-g-1}\\
    &+\mathbb{E}\left[\bigg\|\frac{1}{B}\sum_{\ell=1}^{m_i}(G_i^{j_\ell})^\bot-\bar{\mathbf{G}}_i^\bot\bigg\|^2\right]. \label{eq:var-g-2}
\end{align}
Suppose there are $s_i$ off-diagonal points in $\mathbf{G}_i$. Set $p_i=\frac{s_i}{L}$. Note that $m_i$ can be regarded as a random variable with binomial distribution $m_i\sim \mathrm{Binom}(B,p_i)$. Then we compute 
\begin{equation*}
\begin{aligned}
    \eqref{eq:var-g-1} &= \sum_{k=0}^{B}\mathbb{E}\left[\bigg\|\frac{1}{m_i}\sum_{\ell=1}^{m_i}(G_i^{j_\ell})^\top-\bar{\mathbf{G}}_i^\top\bigg\|^2\bigg|\, m_i=k\right]\mathbb{P}(m_i=k)\\
    &= \sum_{k=1}^{B}\frac{\mathbb{E}[\|(G_i^{j_\ell})^\top-\mathbf{G}_i^\top\|^2]}{k}\mathbb{P}(m_i=k)\\
    &= \bigg(\frac{1}{L}\sum_{\ell=1}^{s_i}\|(G_i^{j_\ell})^\top-\mathbf{G}_i^\top\|^2\bigg)\bigg(\sum_{k=1}^B\frac{\mathbb{P}(m_i=k)}{k}\bigg)\\
    &\le \frac{1}{B}\cdot\bigg(\frac{1}{L}\sum_{\ell=1}^{s_i}\|(G_i^{j_\ell})^\top-\mathbf{G}_i^\top\|^2\bigg)
\end{aligned}
\end{equation*}
where we set the summand to be 0 when $m_i=0$, which corresponds to the degenerate case where all points lie on the diagonal. For \eqref{eq:var-g-2}, by appending $(\partial\Omega)^\bot=0$ to the expression of $\bar{G}_i^\bot$, we have  
\begin{equation*}
\begin{aligned}
    \eqref{eq:var-g-2} &= \mathbb{E}\left[\bigg\|\frac{1}{B}\sum_{\ell=1}^{B}(G_i^{\ell})^\bot-\bar{\mathbf{G}}_i^\bot\bigg\|^2\right]=\frac{1}{B}\mathbb{E}[\|(G_i^{\ell})^\bot-\bar{\mathbf{G}}_i^\bot\|^2]\\
    &=\frac{1}{B}\cdot\bigg(\frac{1}{L}\sum_{\ell=1}^{B}\|(G_i^{\ell})^\bot-\mathbf{G}_i^\bot\|^2\bigg).
\end{aligned}
\end{equation*}
Therefore we have
\begin{equation*}
    \mathbb{E}[\|\bar{G}_i-\bar{\mathbf{G}}_i\|^2]\le \frac{1}{B}\mathbb{V}(\mathbf{G}_i).
\end{equation*}
Summing over all rows, we obtain \eqref{eq:rate-grouping}.
\end{proof}

With a guarantee of uniqueness of Fréchet means for sets of persistence diagrams given by Theorem \ref{thm:unique-mean}, the finite sample convergence rate applies for the empirical Fréchet mean for sets of persistence diagrams exhibiting flat groupings. 

\begin{corollary}
Suppose $\mathbf{G}$ is a flat grouping. Then $D_\star$ is the population Fréchet mean for $D_1,\ldots,D_L$ and $\bar{D}$ is the empirical Fréchet mean for $D'_1,\ldots,D_B'$, and $\mathbb{E}[\mathrm{W}_2^2(\bar{D},D_\star)]\le \frac{\sigma^2}{B}$.
\end{corollary}

Notice that the computations of \eqref{eq:var-g-1} and \eqref{eq:var-g-2} reduce to straightforward vector calculations for flat groupings as $m_i$ is no longer involved as a random variable.  

\section{Alexandrov Geometry of Flat Groupings}
\label{sec:conv}

Recently, \cite{le2022fast} established a general theory on the convergence rate of empirical Fréchet means in Alexandrov spaces with curvature bounded from below, which is the geometric characterization of the space $(\mathcal{S}_2, \mathrm{W}_2)$. However, there are two obstacles to satisfying the assumptions of \cite{le2022fast} in our setting: (i) prior knowledge of whether a given point is Fréchet mean or not; and (ii) a positive lower bound on extensions of geodesics emanating from a Fréchet mean.  This prevents the direct application of the theory by \cite{le2022fast} to understanding the convergence behavior of empirical Fréchet means on $(\mathcal{S}_2, \mathrm{W}_2)$ and we are restricted to the case of flat groupings for convergence behavior.  Nevertheless, it is possible to reconstruct results inspired by a minimal background from \cite{le2022fast} towards a better understanding of the Alexandrov geometry of Fréchet means of persistence diagrams, which is the goal of this section.

\subsection{Metric Properties of Fréchet Means of Persistence Diagrams}

Let $(\mathcal{M},d)$ be a geodesic space. Given two tangent vectors $[u,s],\, [v,t]\in T_z\mathcal{M}$, $[u,s]$ is said to be \emph{opposite} to $[v,t]$ if $s=t=0$ or $s=t\neq 0$ and $\angle_z(u,v)=\pi$. Define
\begin{equation*}
    H_z\mathcal{M} := \{[u,s]\in T_z\mathcal{M} \mid \exists \ [v,t]\in T_z\mathcal{M} \text{ opposite to } [u,s] \}.
\end{equation*}
Let $o_z=[v,0]$ be the tip of the tangent cone. Note that $o_z\in H_z\mathcal{M}$, thus $H_z\mathcal{M}$ is nonempty.
\cite{alexander2022alexandrov} show that $H_z\mathcal{M}$ with the inherited cone metric is in fact a Hilbert space when $\mathcal{M}$ is an Alexandrov space with nonnegative curvature. $H_z\mathcal{M}$ is referred to as the \emph{Hilbert subcone} of the tangent cone at $z$.

Let $\log_z$ be the log map at $z$. Suppose $\log_z(x)=[v_z^x,d(z,x)]$ and $\log_z(y)=[v_z^y,d(z,y)]$. Denote $\langle \log_z(x),\, \log_z(y)\rangle_{z}:=d(z,x)d(z,y)\cos\angle(v_z^x,v_z^y)$.
Let $\mu$ be a probability measure on $(\mathcal{M},d)$ with finite second moment and $z_\star$ be a Fréchet mean of $\mu$. The tangent cone at $z_\star$ then exhibits the following properties.

\begin{theorem}{\cite[Theorem 7]{le2022fast}}\label{thm:le-fast}
Let $(\mathcal{M},d)$ be an Alexandrov space with nonnegative curvature. Then
\begin{enumerate}[(i)]
    \item At $z_\star$, the following equality holds
    \begin{equation}\label{eq:char-barycenter}
        \iint \langle \log_{z_\star}(x),\log_{z_\star}(y)\rangle_{z_\star}\,\mathrm{d}\mu(x)\mathrm{d}\mu(y)=0;
    \end{equation}
    \item The Hilbert subcone at $z_\star$ satisfies $\log_{z_\star}(\mathrm{supp}(\mu))\subseteq H_{\star}\mathcal{M}$;
    \item For any probability measure $\nu$ with finite second moment and $\log_{z_\star}(\mathrm{supp}(\nu))\subseteq H_{z_\star}\mathcal{M}$, and any $y\in\mathcal{M}$,
    \begin{equation}\label{eq:int-hilbert-cone}
        \int_{\mathcal{M}}\langle \log_{z_\star}(x),\, \log_{z_\star}(y)\rangle_{z_\star}\diff{\nu}(x) = \left\langle \int_{H_{z_\star}\mathcal{M}} u \diff{\nu_{\#}}(u),\log_{z_\star}(y)\right\rangle_{z_\star},
    \end{equation}
    where $\nu_\# = (\log_{z_\star})_\#(\nu)$ is the pushforward measure on $H_{z_\star}\mathcal{M}$.
\end{enumerate}
\end{theorem}

For property (i), at any point $z\in\mathcal{M}$ the inequality 
$\displaystyle 
    \iint \langle \log_{z}(x),\, \log_{z}(y)\rangle_{z}\,\mathrm{d}\mu(x)\mathrm{d}\mu(y)\geq 0
$
holds as a consequence of the Lang--Schroeder inequality \citep{le2020note,lang1997kirszbraun}. Furthermore, if $z=z_\star$ is a Fréchet mean, then we have
$\displaystyle
    \int \langle \log_{z_\star}(x),\log_{z_\star}(y)\rangle_{z_\star} \diff{\mu}(x) \leq 0
$
for all $y\in\mathcal{M}$, which yields \eqref{eq:char-barycenter} \citep{le2020note}.

For property (ii), note that although the Hilbert subcone is defined at any point in $\mathcal{M}$, it may be trivial as there may not exist a pair of tangent vectors with opposite directions, as in the space of persistence diagrams.  We formalize this fact below.

\begin{proposition}
The Hilbert subcone at the empty persistence diagram $D_\emptyset$ is trivial with a single point, i.e., $H_{D_\emptyset}(\mathcal{S}_2)=\{o_{D_\emptyset}\}$.
\end{proposition}
\begin{proof}
For any two nonempty persistence diagrams $D_1,D_2$, note that assigning all points to the diagonal gives a trivial bijection between $D_1$ and $D_2$. We have
\begin{equation*}
    \mathrm{W}_2^2(D_1,D_\emptyset)+\mathrm{W}_2^2(D_2,D_\emptyset)\geq \mathrm{W}_2^2(D_1,D_2),
\end{equation*}
meaning that the angle between any two directions at $D_\emptyset$ is bounded by $\frac{\pi}{2}$. Thus, the Hilbert subcone only consists of the tip $o_{D_\emptyset}$.
\end{proof}

   For any $D\in\mathcal{S}_2$, the Hilbert subcone $H_D\mathcal{S}_2$ is a direct sum of $\mathbb{R}^2$ at each off-diagonal points of $D$. In Section \ref{sec:metric-geometry}, we have shown that the tangent vector at $D$ is a collection of vectors $\{v_i\}_{i\in I}\cup\{v_j\}_{j\in J}$ where $I$ is the index set of off-diagonal points and for any $j\in J$, $\frac{v_j}{\|v_j\|}=\frac{1}{\sqrt{2}}(-1,1)$. The tangent vector admits an opposite if and only if $J=\emptyset$.  

Property (ii) of Theorem \ref{thm:le-fast} implies that the Fréchet mean has equally many or more off-diagonal points than any persistence diagram in the support of the probability measure, since the log map sends a persistence diagram to a tangent vector inside the Hilbert subcone at the Fréchet mean.

\begin{definition}
\label{def:hugging}
Fix $z,y\in \mathcal{M}$, the \emph{hugging function} at $z$ with respect to $y$ is defined as
\begin{equation*}
    \kappa_{z}^y(x) = 1-\frac{\mathrm{C}_{z}^2(\log_{z}(x),\log_{z}(y))-d^2(x,y)}{d^2(y,z)}.
\end{equation*}
\end{definition}

Intuitively, the hugging function at $z$ measures the proximity of $\mathcal{M}$ to the tangent cone $T_z\mathcal{M}$. More importantly, at the Fréchet mean, we have the following equality.

\begin{theorem}{\cite[Theorem 8]{le2022fast}}\label{thm:hugging-equality}
Let $(\mathcal{M},d)$ be an Alexandrov space with nonnegative curvature and $z_\star$ be a Fréchet mean for the probability measure $\mu$. Then
\begin{equation}\label{eq:hugging-equality}
    d^2(y,z_\star)\int \kappa_{z_\star}^y(x)\diff{\mu}(x) = \int (d^2(x,y)-d^2(x,z_\star))\diff{\mu}(x)
\end{equation}
for all $y\in\mathcal{M}$.
\end{theorem}


\paragraph{Limitations to Persistence Diagram Space.} The work of \cite{le2022fast} assumes that the hugging function at the barycenter has a positive lower bound for all points in the entirety of the space. This assumption is closely related to the bi-extendibility of geodesics, meaning that a geodesic can be extended for a positive amount of time at both the start and end points. However, in the space of persistence diagrams, no geodesic can extend beyond the diagonal.

\subsection{Convergence of Empirical Fréchet Means of Flat Groupings}

We begin with a computation of the hugging function for flat groupings. The result also gives insight into the terminology of flatness, in the sense of a measure of curvature by the hugging function.

\begin{lemma}\label{lem:hugging-compute}
Let $\bm{\rho}=\frac{1}{L}\sum_{i=1}^LD_i$ be a discrete probability measure and $\mathbf{G}$ be a flat grouping for $D_1,\ldots,D_L$. Suppose $D_1',\ldots,D_B'$ are i.i.d.~samples of $\bm{\rho}$ and $\rho=\frac{1}{B}\sum_{i=1}^LD_i'$ is the empirical measure. Let $D_\star$ be the Fréchet mean of $\bm{\rho}$ and $\bar{D}$ be the Fréchet mean of $\rho$. For any $D_j\in\{D_1,\ldots,D_L\}$, we have
\begin{equation}
    \kappa_{D_\star}^{\bar{D}}(D_j) = \kappa_{\bar{D}}^{D_\star}(D_j) = 1.
\end{equation}
\end{lemma}

\begin{proof}
Let $\Lambda$ be a set of nonnegative numbers $\lambda_1,\ldots,\lambda_L\geq 0$ with $\sum_{j=1}^L\lambda_j=1$, and $\Lambda'$ be another set with the same property. Consider the persistence diagram $D_\Lambda$ such that the off-diagonal points are given by $\mathbf{G}_i^\Lambda = \sum_{i=j}^L\lambda_j\mathbf{G}_i^j$ for every selection $\mathbf{G}_i$. Let $\gamma_\Lambda(t)$ be the geodesic from $D_\star$ to $D_\Lambda$. For any $0\leq t,s\leq 1$, the optimal matching between $\gamma_\Lambda(t)$ and $\gamma_{\Lambda'}(s)$ is given by $t\mathbf{G}_i^\Lambda+(1-t)\bar{\mathbf{G}}_i\mapsto s\mathbf{G}_i^{\Lambda'}+(1-s)\bar{\mathbf{G}}_i$, as they lie within the same cluster. Therefore,
\begin{equation*}
\begin{aligned}
    \cos\angle_{D_\star}&(\log_{D_\star}(D_\Lambda),\,\log_{D_\star}(D_{\Lambda'}))\\
    =& \lim_{t,s\to 0} \frac{\mathrm{W}_2^2(D_\star,\gamma_\Lambda(t))+\mathrm{W}_2^2(D_\star,\gamma_{\Lambda'}(s))-\mathrm{W}_2^2(\gamma_{\Lambda}(t),\gamma_{\Lambda'}(s))}{2\mathrm{W}_2(D_\star,\gamma_\Lambda(t))\mathrm{W}_2(D_\star,\gamma_{\Lambda'}(s))}\\
    =& \lim_{t,s\to 0}\frac{\sum_{i=1}t^2\|\mathbf{G}_i^\Lambda-\bar{\mathbf{G}}_i\|^2+\sum_{i=1}s^2\|\mathbf{G}_i^{\Lambda'}-\bar{\mathbf{G}}_i\|^2-\sum_{i=1}\|t\mathbf{G}_i^\Lambda-s\mathbf{G}_i^{\Lambda'}-(t-s)\bar{\mathbf{G}}_i\|^2}{2ts\sqrt{\dsum_{i=1}\|\mathbf{G}_i^\Lambda-\bar{\mathbf{G}}_i\|^2}\sqrt{\dsum_{i=1}\|\mathbf{G}_i^{\Lambda'}-\bar{\mathbf{G}}_i\|^2}}\\
    =& \frac{\sum_{i=1}\langle \mathbf{G}_i^\Lambda-\bar{\mathbf{G}}_i,\mathbf{G}_i^{\Lambda'}-\bar{\mathbf{G}}_i\rangle}{\sqrt{\dsum_{i=1}\|\mathbf{G}_i^\Lambda-\bar{\mathbf{G}}_i\|^2}\sqrt{\dsum_{i=1}\|\mathbf{G}_i^{\Lambda'}-\bar{\mathbf{G}}_i\|^2}}\\
    =& \frac{\mathrm{W}_2^2(D_\star,D_\Lambda)+\mathrm{W}_2^2(D_\star,D_{\Lambda'})-\mathrm{W}_2^2(D_{\Lambda},D_{\Lambda'})}{2\mathrm{W}_2(D_\star,D_\Lambda)\mathrm{W}_2(D_\star,D_{\Lambda'})}.
\end{aligned}
\end{equation*}
By definition of the cone metric (cf.~\eqref{eq:cone_metric}), we have
\begin{equation*}
\begin{aligned}
    \mathrm{C}_{D_\star}^2&(\log_{D_\star}(D_\Lambda),\,\log_{D_\star}(D_{\Lambda'}))\\
    =&\mathrm{W}_2^2(D_\star,D_\Lambda)+\mathrm{W}_2^2(D_\star,D_{\Lambda'})-2\mathrm{W}_2(D_\star,D_\Lambda)\mathrm{W}_2^2(D_\star,D_{\Lambda'})\cos\angle_{D_\star}(\log_{D_\star}(D_\Lambda),\log_{D_\star}(D_{\Lambda'}))\\
    =&\mathrm{W}_2^2(D_{\Lambda},D_{\Lambda'}),
\end{aligned}
\end{equation*}
which implies that $\kappa_{D_\star}^{D_\Lambda}(D_{\Lambda'})=1$. Specifically, $\kappa_{D_\star}^{\bar{D}}(D_j)=1$ for all $j=1,\ldots,L$.

For the hugging function at $\bar{D}$, a similar computation gives 
\begin{equation*}
    \cos\angle_{\bar{D}}(\log_{\bar{D}}(D_\Lambda),\log_{\bar{D}}(D_{\Lambda'})) = \frac{\mathrm{W}_2^2(\bar{D},D_\Lambda)+\mathrm{W}_2^2(\bar{D},D_{\Lambda'})-\mathrm{W}_2^2(D_{\Lambda},D_{\Lambda'})}{2\mathrm{W}_2(\bar{D},D_\Lambda)\mathrm{W}_2(\bar{D},D_{\Lambda'})}.
\end{equation*}
Thus, the cone metric satisfies $\mathrm{C}_{\bar{D}}^2(\log_{\bar{D}}(D_\Lambda),\log_{\bar{D}}(D_{\Lambda'}))=\mathrm{W}_2^2(D_{\Lambda},D_{\Lambda'})$, which implies $\kappa_{\bar{D}}^{D_\star}(D_j)=1$ for all $j=1,\ldots,L$.
\end{proof}

We are now able to rederive the finite sample convergence rate for flat groupings as in the previous section, but now in the framework of Alexandrov geometry.

\begin{theorem}
Let $\bm{\rho}=\frac{1}{L}\sum_{i=1}^LD_i$ be a discrete probability measure on $\mathcal{S}_2$, and $D_1',\ldots,D_B'$ be i.i.d.~samples. Let $\mathbf{G}$ be a flat grouping for $D_1,\ldots,D_L$. Let $D_\star$ be the population Fréchet mean and $\bar{D}$ be the empirical Fréchet mean with respect the empirical measure $\rho=\frac{1}{B}\sum_{i=1}^LD_i'$. Then
    \begin{equation}
        \mathbb{E}[\mathrm{W}_2^2(\bar{D},D_\star)]\le \frac{\sigma^2}{B},
    \end{equation}
where $\sigma^2=\mathbb{V}(\mathbf{G})$ is the variance of the grouping.
\end{theorem}

\begin{proof}
By Theorem \ref{thm:hugging-equality} and Lemma \ref{lem:hugging-compute}, we have 
$\displaystyle
    \mathrm{W}_2^2(D_\star,\bar{D}) = \int(\mathrm{W}_2^2(D_\star,D)-\mathrm{W}_2^2(\bar{D},D))\diff{\rho}(D).$
By the equality of \eqref{eq:int-hilbert-cone} in Theorem \ref{thm:le-fast}, we have
\begin{equation*}
\begin{aligned}
    2\mathrm{W}_2^2(D_\star,\bar{D}) &= \int(\mathrm{W}_2^2(D_\star,D)+\mathrm{W}_2^2(D_\star,\bar{D})-\mathrm{W}_2^2(\bar{D},D))\diff{\rho}(D)\\
    &= 2\int \mathrm{W}_2(D_\star,D)\mathrm{W}_2(D_\star,\bar{D})\cos\angle_\star(\log_\star D,\log_\star \bar{D})\diff{\rho}(D)\\
    &= 2\int \langle \log_\star D,\log_\star \bar{D}\rangle_\star\diff{\rho}(D)\\
    &= 2\langle \overline{\log_\star D}, \log_\star \bar{D}\rangle_\star,
\end{aligned}
\end{equation*}
where $\overline{\log_\star D}$ denotes the mean of the pushforward empirical measure $(\log_\star)_\#\rho$. Note that $\log_\star(\mathrm{supp}(\rho))\subseteq \log_\star (\mathrm{supp}(\bm\rho))\subseteq H_\star\mathcal{S}_2$. In the Hilbert subcone, we have $\mathrm{W}_2^2(D_\star,\bar{D})\le \mathrm{C}_\star(\overline{\log_\star D},o_\star)\mathrm{C}_\star(\log_\star\bar{D},o_\star)$. Since $\mathrm{C}_\star(\log_\star \bar{D},o_\star)=\mathrm{W}_2(D_\star,\bar{D})$, then $
    \mathbb{E}[\mathrm{W}_2^2(D_\star,\bar{D})]\leq \mathbb{E}[\mathrm{C}_\star^2(\overline{\log_\star D},o_\star)]$.

In Hilbert spaces, we know that the empirical mean $\overline{\log_\star D}$ converges to the population mean $\mathbb{E}[(\log_\star)_\#\bm{\rho}]=o_\star$ in the following sense
\begin{equation*}
    \mathbb{E}[\mathrm{C}_\star^2(\overline{\log_\star D},o_\star)]=\frac{\sigma^2}{B}
\end{equation*}
where 
$ \displaystyle
    \sigma^2 = \int \mathrm{C}_\star^2(\log_\star D,o_\star)]\diff{\rho}(D) = \int \mathrm{W}_2^2(D,D_\star)\diff{\bm{\rho}}(D), 
$
thus completing the proof.
\end{proof}

\section{Approximating Persistent Homology with Fréchet Means of Truncated Persistence Diagrams}
\label{sec:truncatedPD}

Fréchet means of various representations of persistent homology have been used as an approximation for the true persistent homology when the latter is computationally intractable.  In this section, we demonstrate the role that a flat grouping can play in this setting.

Specifically, let $\mathcal{X}$ be a finite yet large metric space with $N$ points. In practice, it is impossible to compute its persistence diagram $D[\mathcal{X}]$ directly; this is due to the the large storage required for all simplices from a filtration. In \cite{roycraft2023bootstrapping,chazal2014stochastic,chazal2015subsampling}, bootstrapping was used to approximate popular vectorizations of the persistence diagrams $D[\mathcal{X}]$, including persistence landscapes, Betti numbers and persistence silhouettes. The procedure is as follows: we first place a probability distribution $\mu$ on the dataset $\mathcal{X}$. We then take $B$ sample sets $X_1,\ldots,X_B$, where each sample set $X_i$ consists of $n\ll N$ i.i.d.~sample points drawn from $\mu$. Let $f$ represent some vectorization of persistence diagrams; we compute $f_i$ for each sample set $X_i$. Then the mean vectorization $\bar{f}=\frac{1}{B}\sum_{i=1}^B f_i$ is a statistical approximation of the true vectorization of $D[\mathcal{X}]$.

Compared to vectorizations of persistence diagrams, the bootstrapping method is a significantly more complicated procedure on the original space of persistence diagrams due to its nonlinearity. \cite{cao2022approximating} propose the Fréchet mean of $D[X_1],\ldots,D[X_B]$ as an approximation of $D[\mathcal{X}]$. Experimentally, the Fréchet mean of a sample persistence diagrams performs well in a bootstrapping procedure. In theory, however, it is not possible to derive a finite sample convergence rate for Fréchet mean since there is no expression for a variance estimation, in general. Thus, in practice, we do not know how to tune the parameters (i.e., determine how many subsamples to take and what their size should be) when bootstrapping and computing Fréchet means to minimize approximation error with an optimal number of sample sets. 

We now show that by cropping points near the diagonal, it is possible to create a flat grouping for persistence diagrams $D[X_1],\ldots,D[X_B]$. Suppose $\mathcal{X}$ is dataset sampled from a $k$-dimensional manifold $\mathcal{M}$. We have the following result by \cite{chazal2014convergence}.
\begin{theorem}{\cite[Theorem 2]{chazal2014convergence}}\label{thm:bottleneck}
    Let $\mathcal{M}$ be a compact $k$-dimensional Riemannian manifold and $\mathcal{X}=\{x_1,\ldots,x_N\}$ be a set of i.i.d.~samples from the uniform distribution on $\mathcal{M}$. Then
    $$
    \mathbb{E}[\mathrm{W}_\infty(D[\mathcal{X}],D[\mathcal{M}])]\le C\bigg(\frac{\log N}{N}\bigg)^{\frac{1}{k}}
    $$
    where $\mathrm{W}_\infty$ is the $\infty$-Wasserstein distance, or bottleneck distance, and $C>0$ is a constant only depending on $\mathcal{M}$.
\end{theorem}

For a compact Riemannian manifold $\mathcal{M}$, we define its separation constant $\lambda(\mathcal{M})$ by
$$
\lambda(\mathcal{M}):=\max\{\lambda: \|p_i-p_j\|>\lambda, \|p_i-\partial\Omega\|>\lambda, \forall\, p_i,p_j\in D[\mathcal{M}] \}.
$$
That is, $\lambda(\mathcal{M})$ is the largest separation of off-diagonal points in $D[\mathcal{M}]$ from each other, and from the diagonal $\partial\Omega$.

In the bootstrapping setup described above, we draw sample sets from the distribution
$$\bm{\rho}=\frac{1}{\binom{N}{n}}\sum_{i=1}\delta_{X_i}$$
on the set of all subsets of $\mathcal{X}$ of cardinality $n$. Suppose $\lambda(\mathcal{M})>0$. If $N>n>1$ is such that
$$
C\bigg(\frac{\log n}{n}\bigg)^{\frac{1}{k}}\le \frac{1}{2}\lambda(\mathcal{M}),
$$
by Theorem \ref{thm:bottleneck}, we know that if $p_\ell\in D[X_i]$ is an off-diagonal point such that $\|p_\ell-\partial\Omega\|<\frac{1}{2}\lambda(\mathcal{M})$, then the optimal matching between $D[X_i]$ and $D[\mathcal{M}]$ will match $p_\ell$ to the diagonal $\partial\Omega$. Thus, for each sample $X_i$, we consider its \emph{truncated persistence diagram},
$$
D_{\mathrm{tr}}[X_i]:=\{p_\ell\in D[X_i]:\|p-\partial\Omega\|>\frac{1}{2}\lambda(\mathcal{M})\}.
$$
This is the persistence diagram obtained by cropping all points between the diagonal and the threshold $\frac{1}{2}\lambda(\mathcal{M})$. The optimal matching between $D[X_i]$ and $D[\mathcal{M}]$ then induces a flat grouping among $D_{\mathrm{tr}}[X_1],\ldots,D_{\mathrm{tr}}[X_{\binom{N}{n}}]$. The Fréchet mean of the truncated persistence diagrams is then unique and approximates the persistence diagram $D[\mathcal{X}]$ for the original dataset $\mathcal{X}$, which we will now demonstarate numerically.

\begin{figure}
    \centering
    \begin{subfigure}{0.45\linewidth}
        \centering
        \includegraphics[width=\linewidth]{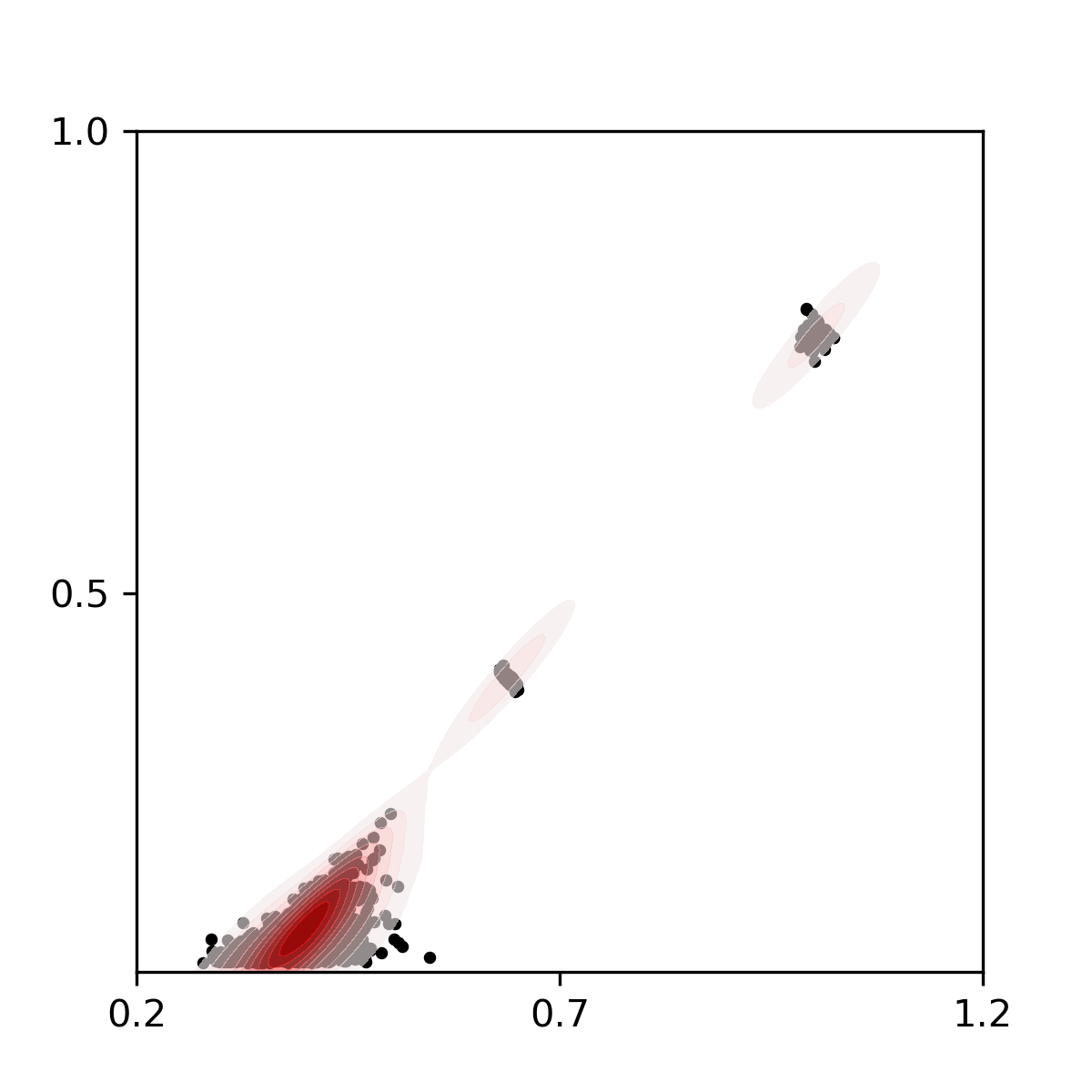}
    \end{subfigure}
    \hfill
    \begin{subfigure}{0.45\linewidth}
        \centering
        \includegraphics[width=\linewidth]{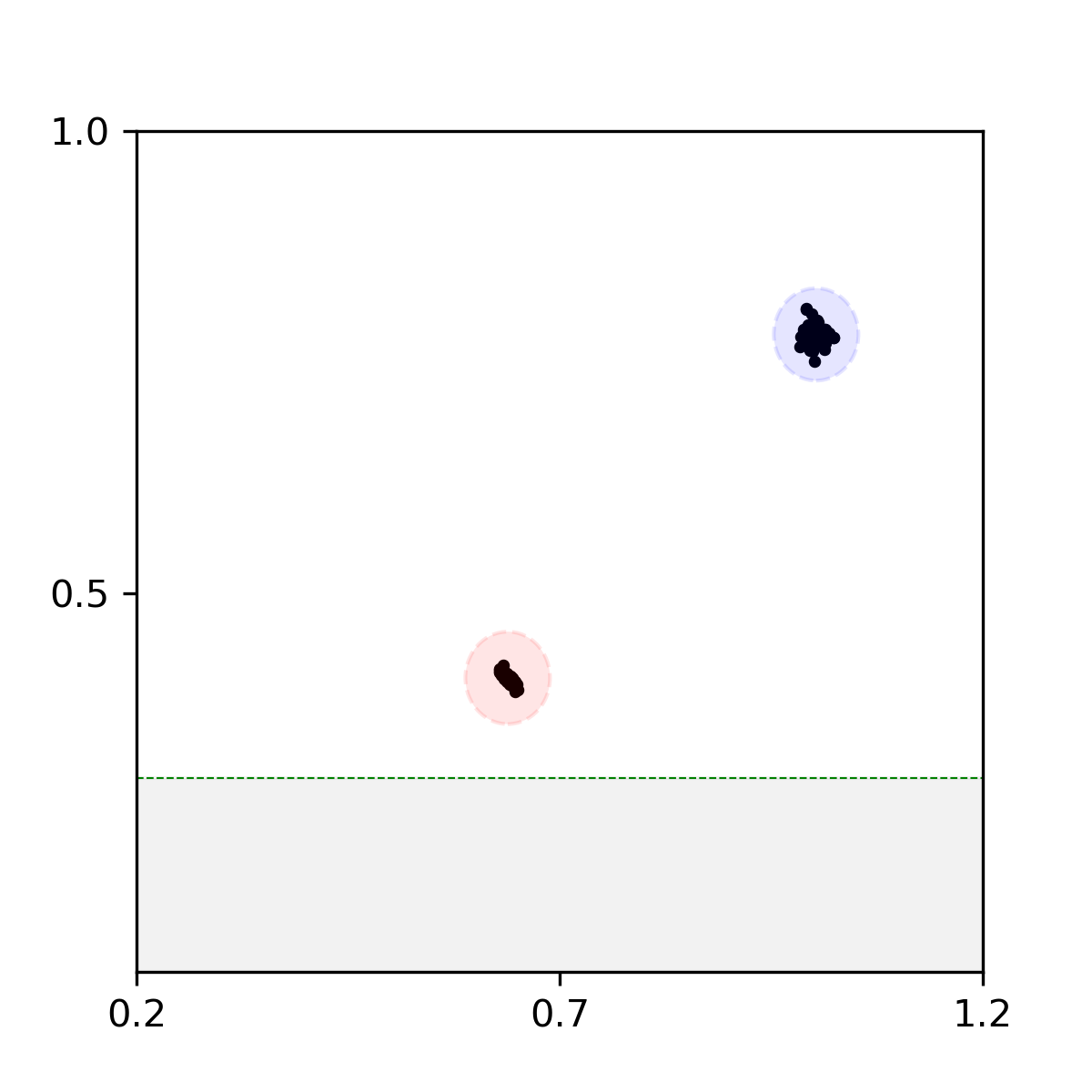}
    \end{subfigure}
    \caption{Illustration of cropping. On the left panel, we have the original persistence diagrams (rotated $45^\circ$ clockwise) computed from 50 sample sets, each consisting of 1000 points from a torus. The shaded background represents the kernel density estimation for off-diagonal points from 50 persistence diagrams. On the right panel, we crop the off-diagonal points below the dotted line. The remaining points form two clusters, thus illustrating that the truncated persistence diagrams form a flat grouping.}
\end{figure}

\paragraph{Circle.}
Take a circle of radius $0.5$. Let $\mathcal{X}$ be a set of 1000 points uniformly sampled from the circle. The 1-dimensional persistence diagram of $\mathcal{X}$ contains $4$ points near the diagonal and one off-diagonal point $(0.0227, 0.8754)$. We then take $B=50$ sample sets $X_1,\ldots,X_{50}$, each consisting of $600$ points from $\mathcal{X}$. We compute the persistence diagrams for $S_i$ and crop the off-diagonal points with a projection to the diagonal being less than $0.2$. The truncated persistence diagrams $D_{\mathrm{tr}}[X_1],\ldots,D_{\mathrm{tr}}[X_{50}]$ form a flat grouping. Their Fréchet mean is given by $\{(0.0395, 0.8582)\}$.

\paragraph{Torus.} 
Take a torus of outer radius $0.8$ and inner radius $0.3$. Let $\mathcal{X}$ be a set of 10,000 points uniformly sampled from the torus. Directly computing the persistent homology of $\mathcal{X}$, we find that the 1-dimensional persistence diagram of $\mathcal{X}$ consists of $(0.0382, 0.5220)$, $(0.0326, 0.8884)$, and $478$ other off-diagonal points that are close to the diagonal. Then we take $B=20$ sample sets $X_1,\ldots,X_{20}$, each consisting of 4000 points from $\mathcal{X}$. We compute the persistence diagrams for $X_i$ and crop the off-diagonal points with a projection to the diagonal being less than $0.3$. The truncated persistence diagrams $D_{\mathrm{tr}}[X_1],\ldots,D_{\mathrm{tr}}[X_{20}]$ form a flat grouping. Their Fréchet mean is given by $\{(0.0597, 0.5222),\, (0.0537, 0.8887)\}$.\\

Code to reproduce these experiments are available at the following repository \url{https://github.com/YueqiCao/CropPD}.

\section{Discussion}
\label{sec:end}

In this paper, we computed a variance expression for a grouping of persistence diagrams and furthermore introduced the notion of flat groupings for sets of persistence diagrams. Flat groupings possess desirable geometric properties that have direct implications on statistical properties in the space of persistence diagrams---specifically, flat groupings generate unique Fréchet means of sets of persistence diagrams.  Additionally, we established convergence for finite samples of groupings and subsequently established convergence for Fréchet means of flat groupings.  We furthermore studied and reinterpreted flat groupings in the general framework of Alexandrov geometry, which contributes to the geometric understanding of the space of persistence diagrams.  On the practical side, we showed that a cropping procedure of persistence diagrams of manifold-valued data gives a construction for flat groupings.  



Our work inspires several directions for future research; we now discuss two important considerations.  Given a set of persistence diagrams, an interesting question is how to systematically determine an appropriate cropping of off-diagonal points to approximate the Fréchet mean. Such a direction would align with existing work that proposes a statistical approach to crop off-diagonal points based on the construction of confidence regions around the diagonal \citep{fasy2014confidence}.  This is a challenging question because more recent work demonstrates that persistence diagrams with many points near the diagonal may in fact correspond to datasets (point clouds) with very clear topological signal \citep{cycleregistration}.  The question of finding an appropriate cropping that preserves ``true'' signal while concurrently allowing for the construction of flat groupings is a new and different direction that requires a deeper exploration beyond existing methods of topological signal processing or cropping.

Another interesting direction considers an alternative measure of centrality of data, namely, the median, which may also be defined for persistence diagrams and has a similar characterization as the Fréchet mean \citep{turner2020medians}. However, understanding this measure entails an entirely different study, since the median is the minimum of the Fréchet function \eqref{eq:frechet-function} with respect to the 1-Wasserstein distance, which would require studying the space $(\mathcal{S}_1,\mathrm{W}_1)$. Less is known about the geometry of $(\mathcal{S}_1,\mathrm{W}_1)$, since it is not any Alexandrov space of curvature bounded from below or above \citep{turner2013means}, thus none of the prior results established in \citep{le2022fast} and used in this work are applicable.  The notion of a Fréchet mean or median can be generalized to cases where the Fréchet functions are of a more general form where $\mathcal{Y}$ is a set and $(\mathcal{M},d)$ is a metric space, and $c:\mathcal{Y}\times \mathcal{M}\to\mathbb{R}$ is a function which gives rise to the $c$-Fréchet mean as the set minimizing $\min_{x\in\mathcal{M}}\mathbb{E}[c(\mathbf{Y},x)]$ where $\mathbf{Y}$ is a $\mathcal{Y}$-valued random variable. The Fréchet mean and median become special cases of $c$-Fréchet means where $c$ is taken to be the square function and the absolute value function. Statistical properties of general $c$-Fréchet means have previously been studied \citep{abd2009estimation,schotz2022strong,electronic},  which then inspires the study of $c$-Fréchet means in the space of persistence diagrams.

\section*{Acknowledgments}

The authors wish to thank Katharine Turner for helpful discussions.  Y.C.~is funded by a President's PhD Scholarship at Imperial College London.  A.M.~is supported by the Engineering and Physical Sciences Research Council under grant reference [EP/Y028872/1], a London Mathematical Society Emmy Noether Fellowship [EN-2223-01], and an Imperial College London Elsie Widdowson Fellowship.


\bibliographystyle{chicago}
\bibliography{ref}

\end{document}